\documentclass[12pt,french,english]{article}
\usepackage{lipsum}

\usepackage[nottoc]{tocbibind}
\usepackage[utf8]{inputenc}
\usepackage[english]{babel}
\usepackage{amsmath}
\usepackage{amsfonts}
\usepackage{amssymb}
\usepackage{amsthm}
\usepackage[normalem]{ulem}
\usepackage{mathtools}
\usepackage{appendix}
\usepackage{tabu}
\usepackage[dvipsnames]{xcolor}
\usepackage{bm}
\usepackage{hyperref}
\usepackage[all,cmtip]{xy}

\usepackage[initials]{amsrefs}
\usepackage{listings}
\usepackage{tikz-cd}
\usepackage[mathscr]{euscript}
\usepackage[left=2.5cm,right=2.5cm,top=2.5cm,bottom=3cm]{geometry}
\lstdefinelanguage{Macaulay2}{
comment=[l]{--},
alsoletter={'},
alsoother={_},
}
\theoremstyle{plain}
\newtheorem{Theorem}{Theorem}[section]
\newtheorem{prop}[Theorem]{Proposition}
\newtheorem{conj}[Theorem]{Conjecture}
\newtheorem{thm}[Theorem]{Theorem}
\newtheorem{cor}[Theorem]{Corollary}
\newtheorem{lem}[Theorem]{Lemma}

\newtheorem*{Ackn}{Acknowledgments}

\theoremstyle{definition}
\newtheorem{ex}[Theorem]{Example}
\newtheorem{dfn}[Theorem]{Definition}

\theoremstyle{remark}

\DeclareMathOperator{\Q}{\mathbb Q}

\DeclareMathOperator{\Co}{\mathbb C}

\DeclareMathOperator{\aut}{Aut}

\DeclareMathOperator{\h}{H}

\DeclareMathOperator{\rk}{rank}
\DeclareMathOperator{\pr}{\mathbb{P}}

\DeclareMathOperator{\ob}{\mathcal{O}}

\makeatletter
\newcommand*{\rom}[1]{\expandafter\@slowromancap\romannumeral #1@}
\makeatother
\begin{document}
\author{Abugaliev Renat}

\title{Characteristic foliation on vertical hypersurfaces in holomorphic symplectic manifolds with Lagrangian Fibration}
\maketitle
\begin{center}\small
    Laboratoire de Mathématiques d’Orsay, Univ. Paris-Sud, CNRS, \\Université Paris-Saclay, 91405 Orsay, France.\\
E-mail address: \textbf{renat.abugaliev@u-psud.fr}

\end{center}

\begin{abstract}
Let $Y$ be a smooth hypersurface in a projective irreducible holomorphic symplectic manifold $X$ of dimension $2n$. The characteristic foliation $F$  is the kernel of the symplectic form restricted to $Y$.  Assume that $X$ is equipped with a Lagrangian fibration $\pi:X\to B$ and $Y=\pi^{-1}D$, where $D$ is a hypersurface in $B$. It is easy to see that the leaves of $F$ are contained in the fibers of $\pi$. We prove that a very general leaf is Zariski dense in a fiber of $\pi$.
\end{abstract}
\section{Introduction}
First, we recall the definition of our main object of study, that is an irreducible holomorphic symplectic manifold.
\begin{dfn}
Let $X$ be a smooth projective variety over $\Co$. It is an irreducible holomorphic symplectic (IHS) if it satisfies the following properties:
\begin{itemize}
    \item $\h^0(X,\Omega^2_X)=\Co \sigma$, where $\sigma$ is a holomorphic symplectic form (at any point of $X$);
    \item $\h^1(X,\ob_X)=0$;
    \item $\pi_1(X)=0$.
\end{itemize}
\end{dfn}

 Note that a holomorphic symplectic form $\sigma$ on a smooth variety $X$ induces an isomorphism between the vector bundles $T_X$ and $\Omega_X$. Indeed, one can map a vector field $v$ to the differential form $\sigma(v,*)$.
\begin{dfn}Let $Y$ be a hypersurface in $X$ and $Y^{sm}$ smooth locus of $Y$. Consider the restriction of $T_X$ to $Y$. The restriction of a symplectic form to any codimension one subspace has rank $2n-2$ i.e. has one-dimensional kernel. The orthogonal complement of the bundle $T_{Y^{sm}}$ in  $T_{X|Y^{sm}}$ is a line subbundle $F$ of $T_{Y^{sm}}\subset T_{X} | _{Y^{sm}}$. We call the rank one subbundle $F\subset T_{Y^{sm}}$ the {\bf characteristic foliation}.
\end{dfn}
Assume $Y$ is smooth. Since $Y=Y^{sm}$, $F$ is a subbundle of $T_Y$. Furthermore, $F$ is isomorphic to the conormal bundle $\mathcal{N}_{Y/X}$ (which is isomorphic to $\ob_Y(-Y)$ by the adjunction formula). Indeed, consider the following short exact sequence: 
$$ 0\to T_Y \to T_X|Y \to \ob_Y(Y) \to 0.$$
Applying the isomorphism $T_X\cong \Omega_X$, we obtain that $F\cong\ob_Y(-Y)$.\\
  Our first question is whether it is algebraically integrable. Jun-Muk Hwang and Eckart Viehweg showed in \cite{HV} that if $Y$ is of general type, then $F$ is not algebraically integrable.  In paper \cite{AC} Ekaterina Amerik and Frédéric Campana  completed this result to the following.
\begin{thm}\cite[Theorem~1.3]{AC} \label{ac}
 Let $Y$ be a smooth hypersurface in an irreducible holomorphic symplectic manifold $X$ of dimension at least $4$. Then  the characteristic foliation on $Y$ is algebraically integrable if and only if $Y$ is uniruled, i.e. covered by rational curves.
\end{thm}

The next step is to ask what could be the dimension of the Zariski closure of a generic leaf of $F$. In dimension $4$ the situation is understood thanks to Theorem \ref{ag}.
\begin{thm}[\cite{AG}]\label{ag}
Let $X$ be an irreducible holomorphic symplectic fourfold and let $Y$ be an irreducible smooth hypersurface in $X$. Suppose that the characteristic foliation $F$ on $Y$ is not algebraically integrable, but there exists a meromorphic fibration on $p:Y \dasharrow C$  by surfaces invariant under $F$ (see Definition \ref{inv}). Then  there exists a rational Lagrangian fibration $X \dasharrow B$ extending $p$. In particular, the Zariski closure of a generic leaf is an abelian surface.
\end{thm}
This leads to the following conjecture. 
\begin{conj}[Campana]\label{conja}
Let $Y$ be a smooth hypersurface in an irreducible holomorphic symplectic manifold $X$ and let $q$ be the Beauville-Bogomolov form on $\h^2(X,\Q)$. Then:
\begin{enumerate}
\item If $q(Y,Y)>0$, a generic leaf of $F$ is Zariski dense in $Y$;
 \item If $q(Y,Y)=0$, the Zariski closure of a generic leaf of $F$ is an abelian variety of dimension $n$;
 \item If $q(Y,Y)<0$, $F$ is algebraically integrable and $Y$ is uniruled.
 \end{enumerate}
 \end{conj} 
Let us explain why this conjecture is plausible and formulate the results we achieved.\\

 \textbf{Case of $q(Y,Y)<0$.} In this case the conjecture is easy to prove. By \cite[Theorem~ 4.2 and Proposition~4.7]{Bo} $Y$ is uniruled, if $q(Y,Y)<0$. There is a dominant rational map $f:Y\dasharrow W$, such that the fibers of $f$ are rationally connected (see \cite{cam} and \cite[Chapter ~IV.5]{KolRC}). Rationally connected varieties do not have non-zero holomorphic differential forms. Thus, the form $\sigma|_Y$ is the pull-back of some form $\omega \in \h^0(W, \Omega^2_W)$ and for any point $x$ in $Y$ the relative tangent space $T_{Y/W,x}$ is  the kernel of the form $\sigma|_Y$. The kernel of $\sigma|_Y$ has dimension one at every point. So, the rational map $f:Y\dasharrow W$ is a fibration in rational curves and these rational curves are the leaves of the foliation $F$.\\

\textbf{Case of $q(Y,Y)=0$}. Conjecturally, $X$ admits a rational lagrangian fibration, and hypersurface $Y$ is the inverse image of a hypersurface of its base (this conjecture was proved for manifolds of K3 type in \cite{BM} and for manifolds of Kummer type \cite{yos}, for O'Grady 6 type in \cite{MR} and for O'Grady 10 type in \cite{MO}). Moreover, a rational  Lagrangian fibration can be replaced with a regular Lagrangian fibration, if we assume that the divisor $Y$ is numerically effective. In this work  we consider an irreducible holomorphic symplectic manifold $X$ equipped with a regular Lagrangian fibration $\pi:X\to B$. We call a preimage $Y$ of divisor $D$ in $B$ vertical. Hwang and Oguiso in \cite{HO} prove that the characteristic foliation on the discriminant hypersurface is algebraically integrable, but this hypersurface is very singular. In this work we consider only smooth vertical hyperspaces.  Our result is as follows.
\begin{thm}\label{t2}Let $X$ be a projective irreducible holomorphic symplectic manifold and \\$\pi:X \to B$ a Lagrangian fibration. Consider a hypersurface $D$ in $B$ such that its preimage $Y$ is a smooth irreducible hypersurface in $X$. Then the closure of a generic leaf of the characteristic foliation on $Y$ is a fiber of $\pi$ (hence an abelian variety of dimension $n$). 
\end{thm}

 \begin{Ackn}The author of this paper is very grateful to Ekaterina Amerik for kind and patient mentorship. I also want to thank Alexey Gorinov, Nikolay Konovalov, Misha Verbitsky, Fabrizio Anella and Daniel Huybrechts for helpful conversations.
 \end{Ackn}

\section{Foliations and invariant subvarieties\label{s2}}
In this section we recall some definitions related to foliations and  preliminary results about the Zariski closure of the leaves. First, we define the foliations and leaves and also mention the Frobenius theorem.
\begin{dfn}Let $X$ be a smooth variety. A (singular) foliation is a saturated
subsheaf $F \subset T_X$ which is closed under the Lie bracket, i.e. $[F, F] \subset F$. The singularity locus $Sing(F)$ of $F$ is the subset of $X$ on which $T_X/F$ is not locally free, and it has codimension at least $2$ in X . A leaf of $F$ is the maximal connected injectively immersed complex analytic submanifold $L \subset X\setminus sing(F)$ such that $T_L =F|_L$.
\end{dfn}
A saturated subsheaf $F \subset T_X$ which does not necessarily satisfy the property $[F, F] \subset F$ is called a distribution. The property $[F, F] \subset F$ is needed for the existence of leafs. 
\begin{thm}[Frobenius]\label{thmfrob} Let $X$ be a smooth variety and $F\subset T_X$ a distribution on $X$. We say that $F\subset T_X$ is integrable if there exists a leaf (i.e. locally closed submanifold $L \subset X\setminus sing(F)$ such that $T_L =F|_L$) through every point of $X\setminus sing(F)$. The distribution $F$ is integrable if and only if $F$ is closed  under the Lie bracket, i.e. $[F, F] \subset F$. 
\end{thm}
\begin{proof}
See for example \cite[Book~I, Theorem~2.20]{Vo}.
\end{proof}
 \begin{dfn}
 If every leaf of a foliation is algebraic we call this foliation {\bf algebraically integrable}.
 \end{dfn}
 In this work we study only foliations of rank one. It follows from theorem \ref{thmfrob} that all distributions of rank one are foliations (i.e. integrable).  
\begin{dfn}\label{inv}  Let $Y$ be  a closed smooth subvariety of $X$. One says it is {\bf invariant under the foliation}  $F$  or {\bf $F$--invariant} if $T_Y$ contains $F|_Y$.
\end{dfn}
The Zariski closure of a leaf through a point $x$ is the smallest invariant under $F$ subvariety containing this point. We denote it by $\overline{Leaf}^{Zar}(x,F)$.
There the following results of Philippe Bonnet from the work \cite{Bo}. He states them for an affine variety $X$. Nevertheless, these statements for an affine $X$ obviously lead to the analogous statements for a projective variety. Thus, let us reformulate them for a projective variety $X$.
\begin{thm}\cite[Theorem~1.3]{bo2} Let $X$ be a projective variety with a foliation $F$.
There is an integer $m$ such that $m$ is equal to the dimension of the Zariski closure of the leaf of $F$ through a very general point $x\in X$. We call this integer $m$ the dimension of the Zariski closure of a generic leaf. Moreover, the dimension of the Zariski closure of the leaf through every point $x\in X$ is not greater than $m$.
\end{thm}
   
\begin{prop}\cite[Theorem~1.4]{bo2}\label{fibr} 
Let $X$ be a smooth projective variety of dimension $n$ with a foliation $F$. Assume that the Zariski closure of a very general leaf of $F$ has dimension $m<n$. Then there exists a rational map $X\dashrightarrow W$ with $F$-invariant fibers of dimension $m$ and a very general fiber of this map is the Zariski closure of a leaf of $F$. 
\end{prop}{}
 \section{Lagrangian fibrations on irreducible holomorphic symplectic manifolds \label{s1}}In this section we define Lagrangian fibrations and recall some results on them.
\begin{thm}[\cite{Ma}] \label{thm3}
Let $X$ be an irreducible  holomorphic symplectic manifold of dimension $2n$  and $\pi: X \to B$  be a regular morphism with connected fibers. Assume that $B$ is a normal variety and $0<\dim B < 2n$. Then:
\begin{itemize}
\item $B$ has dimension $n$;
\item Every fiber of $\pi$ is a {\bf Lagrangian}  subvariety i.e. the restriction of $\sigma$ is zero;
\item Moreover, if a fiber is smooth, it is an abelian variety. 
\end{itemize} 
\end{thm}
\begin{dfn} The morphism $\pi$ as in the  previous theorem is called a {\it Lagrangian fibration}. 
\end{dfn}
\begin{thm}[\cite{Hw}] If $B$ is smooth, then  $B\cong \pr^n$.
\end{thm}
The base $B$ is conjectured to be always smooth. For $n=2$ this conjecture was proved in \cite{HX}.  The authors of \cite{BK} proved this conjecture for all $n$ but for some specific type of singularities of $B$. In the present paper we will not assume that this conjecture is true.\\ 
A Lagrangian subvariety is not always a fiber of a Lagrangian fibration. 
\begin{ex}\label{lagrhilb}

Let $S$ be a $K3$--surface and $C$ be a smooth curve in $S$. Clearly, $C$ is Lagrangian in $S$, but let us construct a more interesting example. Let $Z$ be the strict transform of $C^{(n)}$ in the Hilbert scheme $S^{[n]}$ under the Hilbert-Chow map. The subvariety $Z$ is Lagrangian. Indeed, the preimage of $\sigma_{S^{[n]}}$ in $S^n$ is $p_1^*\sigma_S+...+p_n^*\sigma_S$, where $p_i$ is the $i$--th projection of $S^n$ to $S$. Since $\sigma_S|_C=0$, the restriction of $p_1^*\sigma_S+...+p_n^*\sigma_S$ to $C^{(n)}$ is zero.\end{ex}
\begin{thm}\cite[Theorem~0.4]{SY}\label{cohom}Let $X$ be a projective irreducible holomorphic symplectic variety of dimension $2n$ equipped with a Lagrangian fibration $\pi : X \to B$. Then the restriction of $\h^d(X, \mathbb{Q})\to \h^d(X_b, \mathbb{Q})$ to any nonsingular fiber $X_b\subset X$ is zero for any odd $d$ and has rank one for any even $d$, which is not greater than $2n$. 
\end{thm}

It is well-known that a smooth fiber of a Lagrangian fibration is an abelian variety. Concerning singular fibers there is the following result.
\begin{prop}\cite[Proposition~2.2]{HO}\label{hwang}
Let $\Delta \subset \pr^n$ be the set of points $b \in \pr^n$ such that the fiber $X_b$ is singular. Then:\begin{enumerate}
    \item 
 The set $\Delta$ is a hypersurface in $\pr^n$. We call it the {\bf discriminant hypersurface}.
 \item The normalization of a component of a general fiber of $\pi$ over $\Delta$ is smooth \item 
 The singular locus  of a component of a general singular fiber is a disjoint union of $(n-1)$-dimensional complex tori. 
 \end{enumerate}{}
\end{prop}

\begin{dfn} Let $Y$ be a hypersurface in $X$. If there exists a hypersurface $D$ in the base of a Lagrangian fibration $\pi: X\to B$ such that $Y=\pi^{-1}(D)$, $Y$ is called a {\bf vertical} hypersurface.  
\end{dfn}

Let us give the simplest example of a Lagrangian fibration.
\begin{ex}\label{elk3} Let $S$ be a $K3$ surface with an elliptic fibration $\pi:S \to \pr^1$. This fibration induces a morphism  $$\pi^{(n)}:S^{(n)} \to \pr^n; s_1+s_2+...+s_n \mapsto \pi(s_1)+\pi(s_2)+...+\pi(s_n),$$
where $\pr^n$ is considered as $n$--th symmetric power of $\pr^1$. Composing this morphism with the Hilbert-Chow map, we obtain a morphism with connected fibers from $S^{[n]}$ to $\pr^{n}$. Thus, by Theorem \ref{thm3} it is a Lagrangian fibration. Let $b_1, b_2, ..., b_m\in \pr^1$ be the points such that the fiber $\pi^{-1}(b_i)$ is singular. The discriminant locus of the fibration $\pi$ is the union of the hyperplanes $H_i:=b_i+x_1+x_2+...+x_{n-1}$ and of the hypersurface $\Delta_0:=2x_1+x_2+...+x_{n-1}$, which is tangent to each $H_i$. In particular, for $n=2$ the discriminant hypersurface is the union of the diagonal conic and of the lines tangent to this conic.
\end{ex}

\section{Monodromy of a vertical hypersurface}
To prove theorem \ref{t2} we will need to know some information about the monodromy action on a smooth fiber of a smooth vertical hypersurface. The goal of this section is to prove the following result.
\begin{thm}\label{t1}Let $X$ be a projective irreducible holomorphic symplectic manifold and let $\pi:X \to B$ be a Lagrangian fibration. Consider a hypersurface $D$ in $B$ such that its preimage $Y$ is a smooth irreducible hypersurface in $X$. Let $X_b\subset Y$ be a smooth fiber of $\pi$, then the morphism $\h^2(Y,\Q)\to \h^2(X_b,\Q)$ has $\rk$ one.
\end{thm}
To study the monodromy action on the fibers one should throw away the discriminant divisor $\Delta$ and singular locus $B^{sing}$ of $B$ (in the case $B$ is not smooth). Let us denote $B\setminus (B^{sing}\cup \Delta)$ by $B^o$. Proposition  \ref{cohom} says that $\pi_1(B^o)$ fixes only one-dimensional vector space in $\h^2(X_b,\mathbb{Z})$. Consider a hypersurface $D\subset B$ in the base of the Lagrangian fibration. We denote by $D^o$ the intersection of $B^o$ and $D$. By \cite[Part~II, Section~5.1]{GM} the pushforward of the fundamental group $\pi_1(D^o)\to\pi_1(B^o)$ is surjective $\dim B=2$ and is isomorphism if $\dim B>2$ for a generic $D$. Now we can apply Deligne's invariant cycle theorem (see e.g \cite[Book~II, Theorem~3.2]{Vo}). Let us recall it.
\begin{thm}
Let $X\to Y$ be a smooth projective morphism between smooth quasi-projective varieties.  Then for any point $y\in Y$ and for any integer $k$ the space of the invariants under the monodromy action $\pi_1(Y,y)\to \aut \h^k(X_y,\Q)$ is equal to the image of the restriction map $\h^k(\Bar{X}, \Q) \to \h^k(X_y,\Q)$, for any smooth projective compactification $\Bar{X}$ of $X$\footnote{Here $X_y$ is the fiber $f^{-1}(y).$}.
\end{thm}
For a generic $D$ the map $\pi(D^o)\to \pi(B^o)$ is surjective.  Hence, the groups $\pi(D^o)$ and $\pi(B^o)$ fix the same sublattice in $H^2(X_b,\mathbb{Q})$. Combining this with theorem \ref{cohom} we obtain the following. 
\begin{cor}\label{pr}The statement of theorem \ref{t1} is true for a generic $D$.
\end{cor}
Finally, the following Lemma makes the proof work for an arbitrary smooth vertical hypersurface.
\begin{lem}\label{topp}
Let $X$ be a smooth projective variety and $D_0,D_1$ the linearly equivalent smooth hypersurfaces in $X$. Let $Z$ be a smooth subvariety of $X$ contained in $D_0 \cap D_1$, then $\h^i(D_0,\Q)$ and $\h^i(D_1,\Q)$ have the same images in  $\h^i(Z,\Q)$, for any $i\in \mathbb{Z}$.
\end{lem}
\begin{proof}
Let $I:=[0,1]\hookrightarrow |D_0-Z|:=\pr(\h^0(X,\ob_X(D_0)\otimes \mathcal{J}_Z))$ be a path  between the points $[D_0]$ and $[D_1]$, avoiding the points corresponding to singular hypersurfaces in the linear system $|D_0-Z|$ . Let $\mathcal{D}_I\subset X\times I=\{(x,t)| x\in D_t\}$ and $Z_I=\{(x,t)| x\in Z \}$. Since all hypersurfaces $D_t$ are smooth, $\mathcal{D}_I$ is diffeomorphic to $D_i\times I$ and $Z_I$ is diffeomorphic to $Z\times I$. We have the following commutative squares for all $t\in I$:
 \[
\begin{tikzcd}
Z \arrow{r}{} \arrow{d}{} & Z_I\arrow{d}{}\\
D_t \arrow{r}{} &D_I
\end{tikzcd}
\]
Since the horizontal maps induce an isomorphism on the cohomology groups, the images of the vertical maps are the same.
\end{proof}
Lemma \ref{topp} implies Theorem \ref{t1}: take $D_0=Y$ an arbitrary smooth vertical hypersurface, $Z=X_b$ a smooth fiber of the Lagrangian fibration $\pi$ in $Y$ and $D_1=Y'$ a smooth vertical hypersurface through $X_b$ linearly equivalent to $Y$, such that theorem \ref{t1} is true for $Y'$. Note that a generic hypersurface from the linear system $|Y|$ is smooth by Kleiman's version of the Bertini theorem (\cite{Klei}).
\section{Proof of Theorem \ref{t2} \label{ss2}}
In the last section we apply theorem \ref{t1} to prove theorem \ref{t2}. Let us recall the second.\\

\noindent
\textbf{Theorem \ref{t2}.}
Let $X$ be a projective irreducible holomorphic symplectic manifold and \\$\pi:X \to B$ a Lagrangian fibration. Consider a hypersurface $D$ in $B$ such that its preimage $Y$ is a smooth irreducible hypersurface in $X$. Then the closure of a generic leaf of the characteristic foliation on $Y$ is a fiber of $\pi$ (hence an abelian variety of dimension $n$). \\

Firstly we observe that every fiber of the Lagrangian fibration is invariant under the characteristic foliation. 
\begin{lem} \label{thm4} Let $\pi:X \to B$ be a Lagrangian fibration, let $Y:=\pi^{-1}(D)$ be a smooth irreducible vertical hypersurface in $X$, where $D$ is an irreducible hypersurface in $B$.  Then every fiber of the fibration $\pi: Y \to D$ is  invariant under the characteristic foliation $F$ on $Y$. 
\end{lem}
\begin{proof}
Consider a smooth fiber $X_b$ of the fibration $\pi: Y \to D$ over a point $b\in D$. Let $x$ be a point in $X_b$. The tangent space to $X_b$ at the point $x$ is the orthogonal complement of itself in $T_{X,x}$. Since $T_{D,x}$ contains $T_{X_b,x}$, the space $T_{X_b,x}$ contains the orthogonal complement of $T_{D,x}$ i.e. $F_x$. The singular fibers are invariant as well because of the closedness of this property.\\
\end{proof}
Hence, the Zariski closure of every leaf of the characteristic foliation
on a vertical hypersurface is of dimension at most $n$. Our purpose is to show that this closure is $X_b$ (for very general $b$) and not a proper subvariety of $X_b$.

\begin{prop} \label{thm7} In the assumptions of the Lemma \ref{thm4},  let $Z$ be an irreducible subvariety of a smooth fiber $X_b$, invariant under $F$. Fix a group law on $X_b$, such that $Z$ contains the zero point. For any $a\in X_b$, the translate of $Z$ by a point $a$ is an invariant subvariety. In particular, if $Z$ is a minimal invariant subvariety (i.e. the Zariski closure of a leaf), then it is an abelian variety.

\end{prop}
\begin{proof}
Since the tangent bundle to $X_b$ is trivial, we may view the restriction of the characteristic foliation to $X_b$ as a one-dimensional subspace of $\h^0(X_b,T_{X_b})$. A translation acts trivially on  $\h^0(X_b,T_{X_b})$. Thus, we obtain the first statement. Let $a\in Z$ be a point of $Z$. The translation $Z+a$ is an invariant subvariety. The intersection of $Z+a$ and $Z$ is a non-empty invariant subvariety. Because of the minimality of $Z$, $Z\cap (Z+a)=Z$. In other words, $Z=Z+a$. Hence, $Z$ is an abelian subvariety of $X_b$. 
\end{proof}
\begin{prop}
Assume that the Zariski closure of a generic leaf of the characteristic foliation on $Y$ has dimension less than $n$. Then for a general fiber $X_b$ of $\pi$ the image of the restriction map $\h^2(X,\Q) \to \h^2(X_b,\Q)$ has rank at least $2$. This contradicts Theorem \ref{t1}. 
\end{prop}
\begin{proof}
Consider a rational fibration $p:Y\dashrightarrow W$, which was constructed in Proposition \ref{fibr}.  By Proposition \ref{thm7} its fibers are abelian subvarieties of fibers of $\pi$. Thus, $W \dashrightarrow D$ is a fibration in abelian varieties (quotients of fibers of $\pi$) . Let $G$ be a relatively ample divisor on the fibration $W \dashrightarrow D$ and let $p^{\bullet}G$ be the closure of its preimage in $Y$. The restriction of $p^{\bullet}G$ to a very general fiber is $X_b$ is a pull-back of an ample divisor from  its quotient. Hence, $p^{\bullet}G|_{X_b}$ is effective but not ample. Let $H$ be an ample divisor on $Y$. The restrictions of $G$ and $H$ are not proportional in $\h^2(X_b,\Q)$. 
\end{proof}{}
\bibliography{cf}

\begin{bibdiv}
\begin{biblist}

\bib{AC}{article}{
      author={Amerik, Ekaterina},
      author={Campana, Frederic},
       title={Characteristic foliation on non-uniruled smooth divisors on
  projective hyperkaehler manifolds},
        date={2014},
     journal={Journal of the London Mathematical Society},
      volume={95},
       pages={115\ndash 127},
}

\bib{AG}{article}{
      author={Amerik, Ekaterina},
      author={Guseva, Lyalya},
       title={On the characteristic foliation on a smooth hypersurface in a
  holomorphic symplectic fourfold},
        date={2016},
     journal={Arxiv:1611.00416},
         url={https://arxiv.org/abs/1611.00416},
}

\bib{BM}{article}{
      author={Bayer, Arend},
      author={Macr{\`\i}, Emanuele},
       title={Mmp for moduli of sheaves on k3s via wall-crossing: nef and
  movable cones, lagrangian fibrations},
        date={2013},
     journal={Inventiones mathematicae},
      volume={198},
       pages={505\ndash 590},
}

\bib{BK}{article}{
      author={Bogomolov, Fedor},
      author={Kurnosov, Nikon},
       title={Lagrangian fibrations for ihs fourfolds},
        date={2018},
     journal={arXiv:1810.11011},
}

\bib{bo2}{article}{
      author={Bonnet, Philippe},
       title={Minimal invariant varieties and first integrals for algebraic
  foliations},
        date={2006},
     journal={Bulletin Brazilian Mathematical Society},
      volume={37},
       pages={1\ndash 17},
}

\bib{Bo}{article}{
      author={Boucksom, Sebastien},
       title={Divisorial zariski decompositions on compact complex manifolds},
        date={2004},
     journal={Annales Scientifiques de l{\rq}{\'E}cole Normale Sup{\'e}rieure},
      volume={37},
       pages={45\ndash 76},
}

\bib{cam}{article}{
      author={Campana, F.},
       title={Connexit{\'e} rationnelle des vari{\'e}t{\'e}s de fano},
        date={1992},
     journal={Annales Scientifiques De L Ecole Normale Superieure},
      volume={25},
       pages={539\ndash 545},
}

\bib{GM}{book}{
      author={Goresky, Mark},
      author={MacPherson, Robert},
       title={Stratified morse theory},
   publisher={Springer-Verlag},
        date={1988},
}

\bib{HX}{article}{
      author={Huybrechts, Daniel},
      author={Xu, Chenyang},
       title={Lagrangian fibrations of hyperk{\"a}hler fourfolds},
        date={2019},
     journal={arXiv:1902.10440},
}

\bib{Hw}{article}{
      author={Hwang, Jun-Muk},
       title={Base manifolds for fibrations of projective irreducible
  symplectic manifolds},
        date={2007},
     journal={Inventiones mathematicae},
      volume={174},
       pages={625\ndash 644},
}

\bib{HO}{article}{
      author={Hwang, Jun-Muk},
      author={Oguiso, Keiji},
       title={Characteristic foliation on the discriminant hypersurface of a
  holomorphic lagrangian fibration},
        date={2007},
     journal={American Journal of Mathematics},
      volume={131},
       pages={981\ndash 1007},
}

\bib{HV}{article}{
      author={Hwang, Jun-Muk},
      author={Viehweg, Eckart},
       title={Characteristic foliation on a hypersurface of general type in a
  projective symplectic manifold},
        date={2009},
     journal={Compositio Mathematica},
      volume={146},
       pages={497\ndash 506},
}

\bib{Klei}{article}{
      author={Kleiman, Steven},
       title={The transversality of a general translate},
        date={1974},
     journal={Compositio Mathematica},
      volume={28},
      number={3},
       pages={287\ndash 297},
}

\bib{KolRC}{book}{
      author={Koll{\'a}r, J{\'a}nos},
       title={Rational curves on algebraic varieties},
   publisher={Springer-Verlag},
        date={1996},
}

\bib{Ma}{article}{
      author={Matsushita, Daisuke},
       title={Equidimensionality of {L}agrangian fibrations on holomorphic
  symplectic manifolds},
        date={2000},
     journal={Mathematical Research Letters},
      volume={7},
       pages={389 \ndash  391},
}

\bib{MO}{article}{
      author={Mongardi, Giovanni},
      author={Onorati, Claudio},
       title={Birational geometry of irreducible holomorphic symplectic
  tenfolds of {O}'{G}rady type},
        date={2020},
     journal={arXiv:2010.12511},
}

\bib{MR}{article}{
      author={Mongardi, Giovanni},
      author={Rapagnetta, Antonio},
       title={Monodromy and birational geometry of {O}'{G}rady's sixfolds},
        date={2019},
     journal={ArXiv:1909.07173},
}

\bib{SY}{article}{
      author={Shen, Junliang},
      author={Yin, Qizheng},
       title={Topology of {L}agrangian fibrations and hodge theory of
  hyper-{K}aehler manifolds},
        date={2018},
     journal={ArXiv 1812.10673},
}

\bib{Vo}{book}{
      author={Voisin, Claire},
       title={Hodge theory and complex algebraic geometry},
   publisher={Cambridge University Press},
        date={2003},
}

\bib{yos}{article}{
      author={Yoshioka, Kota},
       title={Bridgeland's stability and the positive cone of the moduli spaces
  of stable objects on an abelian surface},
        date={2012},
     journal={arXiv:1206.4838},
}

\end{biblist}
\end{bibdiv}
\end{document}